\documentclass[12pt,reqno,draft]{article} 
\usepackage{amsmath,amssymb,amsthm,amsfonts, indentfirst}
\usepackage{enumerate,color,bm}
\usepackage{enumitem}
\usepackage{mathtools}
\makeatletter
\def\@cite#1#2{[{{\bfseries #1}\if@tempswa , #2\fi}]}
\renewcommand{\section}{%
\@startsection{section}{1}{\z@}
{0.5truecm plus -1ex minus -.2ex}%
{1.0ex plus .2ex}{\bfseries\large}}
\def\@seccntformat#1{\csname the#1\endcsname.\ }
\makeatother

\numberwithin{equation}{section} 
\theoremstyle{theorem}
\newtheorem{thm}{Theorem}[section]

\newtheorem{lem}[thm]{Lemma}

\theoremstyle{definition}

\newtheorem{remark}{Remark}[section]

\newtheorem*{prth1.1}{Proof of Theorem 1.1}
\newtheorem*{prth1.2}{Proof of Theorem 1.2}
\newtheorem*{prth1.3}{Proof of Theorem 1.3}
\newtheorem*{prcor1.2}{Proof of Corollary 1.2}

\newcommand{\ep}{\varepsilon}
\newcommand{\pa}{\partial}

\newcommand{\R}{\mathbb{R}}

\newcommand{\cl}[1]{{\overline#1}}

\newcommand{\Tmax}{T_{\rm max}}

\begin{document}
\footnote[0]{
    2010{\it Mathematics Subject Classification}\/. 
    Primary: 35B44; 
    Secondary: 35Q92, 92C17.
    }
\footnote[0]{
    {\it Key words and phrases}\/:
    finite-time blow-up; chemotaxis; attraction-repulsion; logistic source.
    }
\begin{center} 
    \Large{{\bf 
    Blow-up phenomena in 
    a parabolic--elliptic--elliptic 
    attraction-repulsion chemotaxis system 
    with superlinear logistic degradation
    }}%
\end{center}
\vspace{5pt}
\begin{center}
    Yutaro Chiyo\footnotemark[1]\footnote[2]{Corresponding author.},\ \ \  
    Monica Marras\footnotemark[3],\ \ \ %
    Yuya Tanaka\footnotemark[1],\ \ \ %
    Tomomi Yokota\footnotemark[1]%
    \footnote[0]{
    E-mail: 
    {\tt ycnewssz@gmail.com},  
    {\tt mmarras@unica.it}, 
    {\tt yuya.tns.6308@gmail.com}, \\ \hspace{17.85mm}
    {\tt yokota@rs.tus.ac.jp}
    }\\
    \vspace{25pt}
    {\small 
    \footnotemark[1]Department of Mathematics, \\
    Tokyo University of Science\\
    1-3, Kagurazaka, Shinjuku-ku, 
    Tokyo 162-8601, Japan}\\
    \vspace{18pt}
    {\small 
    \footnotemark[3]Department of Mathematics and Computer Sciences, \\
    University of Cagliari\\
    via Ospedale 72, 09123 Cagliari, Italy}\\
    \vspace{2pt}
\end{center}
\begin{center}    
    \small \today
\end{center}

\vspace{2pt}
\newenvironment{summary}
{\vspace{.5\baselineskip}\begin{list}{}{%
     \setlength{\baselineskip}{0.85\baselineskip}
     \setlength{\topsep}{0pt}
     \setlength{\leftmargin}{12mm}
     \setlength{\rightmargin}{12mm}
     \setlength{\listparindent}{0mm}
     \setlength{\itemindent}{\listparindent}
     \setlength{\parsep}{0pt}
     \item\relax}}{\end{list}\vspace{.5\baselineskip}}
\begin{summary}
{\footnotesize {\bf Abstract.} 
 This paper is concerned with 
the attraction-repulsion 
chemotaxis system with superlinear logistic degradation,
%
\begin{align*}
\begin{cases}
u_t = \Delta u  - \chi \nabla\cdot(u \nabla v) 
        + \xi \nabla\cdot (u \nabla w) + \lambda u - \mu u^k,  \quad 
&x \in \Omega,\ t>0,\\[1.05mm]
0= \Delta v + \alpha u - \beta v, \quad 
&x \in \Omega,\ t>0,\\[1.05mm]
0= \Delta w + \gamma u - \delta w, \quad 
&x \in \Omega,\ t>0,
\end{cases}
\end{align*}
%
under homogeneous Neumann boundary conditions, 
in a ball $\Omega \subset \R^n$ ($n \ge 3$), 
with constant parameters $\lambda \in \R$, $k>1$,  
$\mu, \chi, \xi, \alpha, \beta, \gamma, \delta>0$. 
Blow-up phenomena in the system have been well investigated 
in the case $\lambda=\mu=0$, 
whereas the attraction-repulsion chemotaxis system with logistic degradation 
has been not studied. 
Under the condition that $k>1$ is close to $1$, 
this paper ensures a solution 
which blows up 
in $L^\infty$-norm and $L^\sigma$-norm 
with some $\sigma>1$  
for some nonnegative initial data. 
Moreover, a lower bound of blow-up time is derived. 
}
\end{summary}

\vspace{10pt}

\newpage

\section{Introduction} \label{Sec1}

{\it Chemotaxis} is a property of cells to move in 
response to the concentration gradient of a chemical substance 
produced by the cells.  
More precisely, it accounts for a process in which cells exhibit in response 
to chemoattractant and chemorepellent which are produced by themselves, 
that is, moving towards higher concentrations of an attractive signal 
and keeping away from a repulsive signal. 
A fully parabolic attraction-repulsion chemotaxis system was proposed by 
Painter and Hillen~\cite{PH-2002} to show the quorum effect 
in the chemotactic process and 
Luca et al.~\cite{LCEM-2003}
to describe the aggregation of microglia observed in Alzheimer's disease, 
and can be approximated by a parabolic--elliptic--elliptic system. 
\vspace{2mm}

In this paper we consider the parabolic--elliptic--elliptic 
attraction-repulsion chemotaxis system with superlinear logistic degradation,
%
\begin{align}\label{sys1}
\begin{cases}
u_t = \Delta u  - \chi \nabla\cdot(u \nabla v) 
        + \xi \nabla\cdot (u \nabla w) + \lambda u - \mu u^k,  \quad 
&x \in \Omega,\ t>0,\\[1.05mm]
0= \Delta v + \alpha u - \beta v, \quad 
&x \in \Omega,\ t>0,\\[1.05mm]
0= \Delta w + \gamma u - \delta w, \quad 
&x \in \Omega,\ t>0,\\[1.05mm]
\frac{\pa u}{\pa \nu}=\frac{\pa v}{\pa \nu}=\frac{\pa w}{\pa \nu}=0,
&x \in \pa\Omega,\ t>0,\\[1.05mm]
u(x,0)=u_0(x),
&x \in \Omega,
\end{cases}
\end{align}
%
where $\Omega:=B_R(0) \subset \R^n$ $(n \ge 3)$ is an open ball 
centered at the origin with radius $R>0$; 
$\lambda \in \R$, $k>1$ and 
$\mu, \chi, \xi, \alpha, \beta, \gamma, \delta$ are positive constants; 
$\frac{\pa}{\pa \nu}$ is the outward normal derivative on $\pa \Omega$. 
Moreover, the initial data $u_0$ is supposed to satisfy
%
\begin{align}\label{initial}
u_0 \in C^0(\cl{\Omega})
\text{ is radially symmetric and nonnegative.}
\end{align}
%

%
\noindent
The functions $u$, $v$ and $w$ represent the cell density, the concentration of attractive 
and repulsive chemical substances, respectively. 
\vspace{2mm}

Blow-up phenomena correspond to the concentration of organisms 
on chemical substances. 
Hence it is important to investigate whether 
a solution of system \eqref{sys1} blows up or not. 
In this paper we show finite-time blow-up in $L^\infty$-norm and 
$L^\sigma$-norm with some $\sigma>1$, 
and derive a lower bound of blow-up time.
Still more, not only blow-up phenomena but also global existence and 
boundedness have been studied in many literatures 
on chemotaxis systems (see \cite{HP-2009}, \cite{BBTW-2015} and \cite{AT}). 
Before presenting the main results, we give an overview of 
known results about some problems related to \eqref{sys1}.
\vspace{2mm}

We first focus on the chemotaxis system
%
\begin{align}\label{pplog}
\begin{cases}
      u_t=\Delta u-\chi\nabla \cdot (u\nabla v)+g(u),
    \\[1.05mm]
    \tau  v_t=\Delta v+\alpha u-\beta v
\end{cases}
\end{align}
%
under homogeneous Neumann boundary conditions, where  
$\chi, \alpha, \beta$ are positive constants and $g$ is a function of logistic type, 
$\tau\in\{0,1\}$.
The system with $g(u) \equiv 0$ was proposed 
by Keller and Segel~\cite{KS-1970}. 
Since then, system \eqref{pplog} 
was extensively investigated as listed below.
\vspace{3mm}
\begin{itemize}
\item 
If $\tau=1$, $g(u) \equiv 0$ and $\alpha=\beta=1$, global existence and boundedness as well as finite-time blow-up were investigated as follows. 
In the one-dimensional setting, Osaki and Yagi~\cite{OY-2001} showed that 
all solutions are global in time and bounded. 
In the two-dimensional setting, 
Nagai et al.~\cite{NSY-1997} established 
global existence and boundedness under the condition $\int_\Omega u_0(x)\,dx<\frac{4\pi}{\chi}$. 
On the other hand, Herrero and Vel\'{a}zquez~\cite{HV-1997} presented 
existence of radially symmetric solutions which blow up in finite time.  
Winkler in \cite{W-2010} with $\chi=1$ and $n\geq 3$,  derived that if
$\|u_0\|_{L^{\frac{n}{2}+\ep}(\Omega)}$ and $\|\nabla v_0\|_{L^{n+\ep}(\Omega)}$ 
are small for sufficiently small $\ep>0$, 
then a solution is global and bounded. 
Also, Winkler in \cite{W-2013} proved finite-time blow-up under some conditions for initial data $(u_0, v_0)$. 
\vspace{4mm}
\item 
If $\tau=1$ and $g(u)=\lambda u-\mu u^k$ with $\lambda, \mu>0$, 
global existence for any $k>1$ and 
stabilization for $k \ge 2-\frac{2}{n}$ 
were achieved in a generalized solution concept by Winkler~\cite{W-2020-4}. 
Also, for certain choices of $\lambda, \mu$, Yan and Fuest in \cite{YF}, derived global existence of weak solutions 
under the condition $k>\min\{2-\frac{2}{n},\ 2-\frac{4}{n+4}\},$ $n\geq 2$ 
and $\alpha=\beta=1$. 
In particular  for $n=2$, they showed that taking any $k>1$ suffices to exclude the possibility of collapse into a persistent Dirac distribution.
\vspace{4mm}
\item
If $\tau=0$, $g(u) \equiv 0$ and $\beta=1$, Nagai in \cite{N-1995} 
proved global existence and boundedness when $n=1$, or $n=2$ and
$\int_\Omega u_0(x)\,dx<\frac{4\pi}{\chi\alpha}$, 
and finite-time blow-up under some condition for the energy function and the moment of $u$ when $n \ge 2$. 
Also, in the two-dimensional setting, 
Nagai in \cite{N-2001} obtained 
global existence and boundedness 
under the condition $\int_\Omega u_0(x)\,dx<\frac{8\pi}{\chi\alpha}$, 
and finite-time blow-up 
under the conditions that
$\alpha=1$, $\int_\Omega u_0(x)\,dx>\frac{8\pi}{\chi}$ and that 
\begin{align}\label{intsmall}
\int_\Omega u_0(x)|x-x_0|^2\,dx\ 
\text{is sufficiently small for some $x_0 \in \Omega$.}
\end{align}
\item
If $\tau=0$, $g(u) \le a-\mu u^2$ with $a>0$, 
$\mu>0$ ($n \le 2$), $\mu>\frac{n-2}{n}\chi$ ($n \ge 3$) 
and $\alpha=\beta=1$, 
Tello and Winkler in \cite{TW-2007} showed global existence and 
boundedness.
\vspace{4mm}
 \item
If $\tau=0$ and $\chi=\alpha = \beta=1$, when $g(u)=\lambda u-\mu u^k$ with $\lambda \in \R$, $\mu>0$ and $k>1$, 
Winkler in \cite{W-2018} established finite-time blow-up 
in $L^\infty$-norm under suitable conditions on data; 
more precisely, the author asserted that 
if $\Omega=B_R(0) \subset \R^n$ with $n \ge 3$, $R>0$ 
and $1<k<\frac{7}{6}$ $(n \in \{3,4\})$, 
$1<k<1+\frac{1}{2(n-1)}$ $(n \ge 5)$, 
then system \eqref{pplog} admits 
a solution which blows up in 
$L^\infty$-norm at finite time. 
In~\cite{MV-2020}, Marras and Vernier derived finite-time blow-up in $L^\sigma$-norm with $\sigma>\frac{n}{2}$ 
and finally obtained a lower bound of blow-up time. 
Moreover, as to system \eqref{pplog} 
with nonlinear diffusion, 
finite-time blow-up in $L^\infty$-norm 
was obtained by Black et al.\ in \cite{BFL} 
(see also \cite{T}, \cite{TY-2020} 
for weak chemotactic sensitivity and \cite{MNV-2020}  for finite-time blow-up
 in $L^p$-norm to more general chemotaxis system). 
\end{itemize}
\vspace{4mm}

We now shift our attention to the attraction-repulsion chemotaxis system 
%
\begin{align}\label{pee}
\begin{cases}
      u_t=\Delta u-\chi\nabla \cdot (u\nabla v)
            +\xi \nabla\cdot(u \nabla w) + g(u),
    \\[1.05mm]
    \tau v_t  =\Delta v+\alpha u-\beta v,
    \\[1.05mm]
      \tau w_t=\Delta w+\gamma u-\delta w
\end{cases}
\end{align}
%
under homogeneous Neumann boundary conditions, 
where $\chi, \xi, \alpha, \beta, \gamma, \delta>0$ are constants
and $\tau\in\{0,1\}$. 
The system with $\tau=0$ and 
$g(u)=\lambda u-\mu u^k$ 
coincides with \eqref{sys1}, 
whereas the previous works on 
this system are collected as follows.
\vspace{4mm}
\begin{itemize}
\item
If $\tau=0$ and $g(u)\equiv 0$, existence of solutions which blow up in $L^\infty$-norm at finite time 
was studied in~\cite{TW-2013} and \cite{LL-2016}. 
More precisely, in the two-dimensional setting, 
Tao and Wang~\cite{TW-2013} derived finite-time blow-up 
under the conditions \eqref{intsmall} and 
%
\begin{itemize}
\item[] \hspace{-8mm} (i)
     $\chi\alpha-\xi\gamma>0$, \quad $\delta=\beta$\quad and 
     \quad $\int_\Omega u_0(x)\,dx>\frac{8\pi}{\chi\alpha-\xi\gamma}$.
\end{itemize}
%
Also, in the two-dimensional setting, Li and Li~\cite{LL-2016} 
extended the above (i) to the following two conditions:
%
\begin{itemize}
\item[] \hspace{-8mm} (ii)
     $\chi\alpha-\xi\gamma>0$, \quad $\delta \ge \beta$\quad and 
     \quad $\int_\Omega u_0(x)\,dx>\frac{8\pi}{\chi\alpha-\xi\gamma}$;
\item[] \hspace{-8mm} (iii)
     $\chi\alpha\delta-\xi\gamma\beta>0$, \quad $\delta<\beta$ \quad 
     and \quad 
     $\int_\Omega u_0(x)\,dx>\frac{8\pi}{\chi\alpha\delta-\xi\gamma\beta}$.
\end{itemize}
%
\vspace{4mm}
\item
If $\tau=0$ and $g(u) \equiv 0$, Yu et al.~\cite{YGZ-2017} 
replaced $\chi\alpha\delta-\xi\gamma\beta$ with $\chi\alpha-\xi\gamma$ in (iii) 
and filled the gap between the above (ii) and (iii). 
In~\cite{TW-2013}, \cite{LL-2016} and \cite{YGZ-2017}, 
blow-up phenomena were analyzed 
by introducing the linear combination of the solution components $v, w$ 
such that $z:=\chi v-\xi w$ (as to the fully parabolic case $\tau=1$, 
see~\cite{CY-2021}). 
\vspace{4mm}
\item 
If  $\tau=0$ and $g(u)\equiv 0$, explicit lower bound of blow-up time for system \eqref{pee} 
was provided under the condition $\chi\alpha-\xi\gamma>0$ in the two-dimensional setting (see \cite{V-2019}). 
\end{itemize}
\vspace{4mm}

%
In summary, 
blow-up phenomena have been well studied in both a parabolic--elliptic 
Keller--Segel system and an attraction-repulsion one 
when logistic sources are missing.  
However, 
blow-up with effect of logistic degradation in a 
Keller--Segel system has been investigated, 
while for an attraction-repulsion system 
it is still an open problem. 
\vspace{2mm}

The purpose of this paper is to solve the above open problem. 
Namely, we examine finite-time blow-up 
in the attraction-repulsion system \eqref{sys1}  
and we achieve a lower bound of the blow-up time. 
%
\vspace{2mm}

We now state main theorems. 
The first one asserts finite-time blow-up in $L^\infty$-norm. 
The statement reads as follows. 

%
%
\begin{thm}[Finite-time blow-up in $L^\infty$-norm]\label{BULinfty}
Let\/ $\Omega=B_R(0) \subset \R^n$, $n \ge 3$ and 
$R>0$, and let $\lambda \in \mathbb{R}$, $\mu>0$, 
$\chi, \xi, \alpha, \beta, \gamma, \delta>0$. 
Assume that $k>1$ satisfies
\begin{align}\label{condik}
k<
\begin{cases}
\frac{7}{6} & {\it if}\ n \in \{3,4\},\\
1+\frac{1}{2(n-1)} & {\it if}\ n \ge 5, 
\end{cases}
\end{align}
and $\chi, \xi, \alpha, \gamma>0$ fulfill $\chi\alpha-\xi\gamma>0$. 
Then, for all $L>0$, $m>0$ and $m_0 \in (0, m)$ one can find 
$r_0=r_0(R, \lambda, \mu, k, L, m, m_0) \in (0, R)$ 
with the property that whenever $u_0$ satisfies
\eqref{initial} and is such that
\begin{align}\label{u0sing}
u_0(x) \le L|x|^{-n(n-1)}\quad {\it for\ all}\ x \in \Omega
\end{align}
as well as
\begin{align*}
\int_\Omega u_0(x)\,dx \le m\quad {\it but}\quad \int_{B_{r_0}(0)}u_0(x)\,dx \ge m_0,
\end{align*}
there exist $\Tmax \in (0,\infty)$ and a classical solution $(u, v ,w)$ of 
system \eqref{sys1}, uniquely determined by 
\begin{align*}
    &u\in C^0(\overline{\Omega}\times[0,\Tmax))\cap C^{2,1}(\overline{\Omega}\times(0,\Tmax)),\\[1mm]
    &v, w\in \bigcap_{\vartheta>n}C^0([0,\Tmax);W^{1,\vartheta}(\Omega))\cap C^{2,1}(\overline{\Omega}\times(0,\Tmax)),
\end{align*}
which blows up at $t=\Tmax$ in the sense that 
\begin{align}\label{blowupinfty}
\limsup_{t \nearrow \Tmax} \|u(\cdot, t)\|_{L^\infty(\Omega)}=\infty.
\end{align}
\end{thm}
%

We next state a result which guarantees a solution blows up in $L^\sigma$-norm at the blow-up time in $L^\infty$-norm.  
The theorem is the following.
%
%
\begin{thm}[Finite-time blow-up in $L^\sigma$-norm]\label{BULsig}
Let\/ $\Omega = B_R(0) \subset\mathbb{R}^n$, $n\geq 3$ and $R>0$. 
Then, a classical solution $(u, v, w)$ for 
$t \in (0, \Tmax)$,
provided by Theorem~\ref{BULinfty}, is such that for all 
$\sigma>\frac{n}{2}$, 
\begin{align*}
\limsup_{t \nearrow T_{\rm max}}\, 
\left\|u(\cdot,t)\right\|_{L^{\sigma}(\Omega)} 
= \infty.
\end{align*}
\end{thm}

Define for all  $\sigma>1$ the energy function
\begin{equation}\label{Psi}
\Psi(t):= \frac 1 {\sigma} \|u(\cdot,t)\|^{\sigma}_{L^{\sigma}(\Omega)} \quad {\rm with}\quad  \Psi_0 := \Psi(0)= \frac 1 {\sigma} \|u_0\|^{\sigma}_{L^{\sigma}(\Omega)}.
\end{equation}
%

The third theorem provides a lower bound of blow-up time. 
The result reads as follows.
%
%
\begin{thm}[Lower bound of blow-up time]\label{LB}
Let\/ $\Omega = B_R(0) \subset\mathbb{R}^n$, $n\geq 3$ 
and $R>0$. 
Then,
for all  $\sigma>\frac n2$ and some $B_1\ge0$, $B_2, B_3>0$
depending on $\lambda$, $\mu$, $\sigma$, and $n$, 
the blow-up time $\Tmax$, provided by Theorem~\ref{BULinfty}, 
satisfies the estimate
\begin{align}\label{lower Tmax in Lp}
\Tmax\geq \int_{\Psi_0}^{\infty}\frac{d\eta}{B_1 \eta + B_2\eta^{\gamma_1} + B_3 \eta^{\gamma_2}},
  \end{align}
with $\gamma_1:= \frac{\sigma+1}{\sigma}$, 
$\gamma_2:=\frac{2(\sigma+1)-n} {2\sigma-n}.$
\end{thm}
%

One of the difficulties in the proofs of 
the above theorems is that 
the transformation $z:=\chi v- \xi w$ 
does not work to reduce \eqref{sys1} to 
the Keller--Segel system in the case 
$\beta \neq \delta$, 
in contrast to the case $\beta=\delta$ 
which ensures the simplification of 
\eqref{sys1} as
\begin{align*}
\begin{cases}
u_t = \Delta u  - \nabla\cdot(u \nabla z) + \lambda u - \mu u^k,  \quad &x \in \Omega,\ t>0,\\[1.05mm]
0= \Delta z + (\chi\alpha - \xi \gamma) u -\beta z,  \quad &x \in \Omega,\ t>0,
\end{cases}
\end{align*}
which has already been studied in 
\cite{MV-2020, W-2018}. 
To overcome the difficulty, 
we carry out the arguments in the literatures 
without using the above transformation $z$. 
In particular, we need to handle 
the effect caused by the repulsion term 
$\xi\nabla \cdot (u\nabla w)$. 
\vspace{2mm}

This paper is organized as follows. 
In Section~\ref{Sec2} we give 
preliminary results on local existence 
of classical solutions to \eqref{sys1} and 
some basic and useful facts. 
In Sections~\ref{Sec3} and \ref{Sec4} we prove finite-time blow-up 
in $L^\infty$-norm and $L^\sigma$-norm (Theorems~\ref{BULinfty} 
and \ref{BULsig}), respectively. 
Section~\ref{Sec5} is devoted to
the derivation of
a lower bound of 
blow-up time (Theorem~\ref{LB}). 


\section{Preliminaries} \label{Sec2}

We start with the following lemma 
on local existence of 
classical solutions to \eqref{sys1}. 
This lemma can be proved by a standard fixed point argument 
(see e.g.,~\cite{TW-2007}).
%
\begin{lem} \label{LSE}
Let\/ $\Omega=B_R(0) \subset \R^n$, $n \ge 3$ and 
$R>0$, and let $\lambda \in \mathbb{R}$, $\mu>0$, $k>1$, 
$\chi, \xi, \alpha, \beta, \gamma, \delta>0$.
 Then for all nonnegative $u_0 \in C^0(\overline{\Omega})$ there exists $\Tmax \in (0,\infty]$ such that 
 \eqref{sys1} possesses a unique classical solution 
 $(u, v, w)$ such that
%
\begin{align*}
    &u\in C^0(\overline{\Omega}\times[0,\Tmax))
           \cap C^{2,1}(\overline{\Omega}\times(0,\Tmax)),\\[1.05mm]
    &v, w\in \bigcap_{\vartheta>n}C^0([0,\Tmax);W^{1,\vartheta}(\Omega))
           \cap C^{2,1}(\overline{\Omega}\times(0,\Tmax)),
\end{align*}
%
 and
%
    \begin{align*}
            &u\ge0,
    \quad 
            v \ge 0,
    \quad 
            w \ge 0
    \quad 
            {\it for\ all}\ t \in (0, \Tmax).
    \end{align*}
%
 Moreover, 
%
    \begin{align} \label{bucri}
            {\it if}\ \Tmax<\infty,
    \quad 
            {\it then}\ \limsup_{t \nearrow \Tmax} 
                           \|u(\cdot,t)\|_{L^\infty(\Omega)}=\infty.
    \end{align}
%
\end{lem}
%
%
\begin{remark}
We can use $\lim_{t \nearrow \Tmax} \|u(\cdot,t)\|_{L^\infty(\Omega)}$ instead of $\limsup_{t \nearrow \Tmax} \|u(\cdot,t)\|_{L^\infty(\Omega)}$ in the blow-up criterion \eqref{bucri}, 
because we can construct a classical solution on $[0,T]$ 
with some positive time $T$ depending only on $\|u_0\|_{L^\infty(\Omega)}$ 
and discuss the extension of the classical solution 
in a neighborhood of its maximal existence time $\Tmax$,
if $\Tmax<\infty$. 
\end{remark}
%

We next give some properties of the Neumann heat semigroup 
which will be used later. 
For the proof, see \cite[Lemma 2.1]{Cao} and \cite[Lemma~1.3]{W-2010}.
\newpage
\begin{lem} \label{Cao} 
Suppose $(e^{t \Delta})_{t\geq 0}$ is the Neumann heat semigroup in 
$\Omega$, and let $\mu_1 >0$ denote the first non zero eigenvalue of 
$-\Delta$ in $\Omega$ under Neumann boundary conditions. 
Then there exist $k_1,  k_2 >0$ which 
only depend on $\Omega$ and have  
the following properties\/{\rm :}
\begin{enumerate}[label=(\roman*)]
\item if\/ $1 \leq q\leq p\leq \infty$, then 
\begin{equation} \label{etDeltaz}
\|e^{t \Delta} z\|_{L^{p}(\Omega)} \leq k_1t^{ - \frac n 2(\frac 1 q - \frac 1 p)} 
\|z\|_{L^q(\Omega)}, \ \ \forall \ t >0
\end{equation}
holds for all $z\in L^q(\Omega)$.
\item If\/ $1< q \leq p \leq \infty$, then
\begin{equation} \label{etDelta nablaz}
\|e^{t \Delta} \nabla \cdot \textbf{z\,}\|_{L^{p}(\Omega)} \leq k_2\big(1+ t^{-\frac 1 2 - \frac n 2(\frac 1 q - \frac 1 p)}\big) e^{-\mu_1 t} \|\textbf{z\,}\|_{L^q(\Omega)}, \ \ \forall \ t >0
\end{equation}
is valid for any $\textbf{z\,} \in (L^{q}(\Omega))^n$, 
where $e^{t \Delta} \nabla \cdot {}$ is the extension of the operator
$e^{t \Delta} \nabla \cdot {}$ on $(C_0^\infty(\Omega))^n$ to $(L^q(\Omega))^n$. 
\end{enumerate}
\end{lem}

In Section~\ref{Sec5} we will use the following lemma which is obtained by a minor adjustment 
of the power of the Gagliardo--Nirenberg inequality.
\begin{lem}\label{lemma GN ineq}
Let\/ $\Omega$ be a  bounded and smooth domain of\/ $\mathbb{R}^n$ with $n\geq 1$. 
Let $ \mathsf{r}\geq 1$, $0< \mathsf{q} \leq  \mathsf{p}\leq\infty$, $ \mathsf{s}>0$.
	Then there exists a constant
	$C_{{\rm GN}}>0$ such that 
\begin{equation} \label{GN ineq}	
		\|f\|^{\mathsf{p}}_{L^{ \mathsf{p}}(\Omega)}\leq C_{{\rm GN}} \Big(\|\nabla f\|^{\mathsf{p} a}_{L^{ \mathsf{r}}(\Omega)}
		\|f\|_{L^{ \mathsf{q}}(\Omega)}^{{\mathsf{p}}(1-a)}
	+\|f\|^{\mathsf{p}}_{L^{ \mathsf{s}}(\Omega)}\Big)
	\end{equation}
for all $f\in L^{\textsf{q}}({\Omega})$ with 
$\nabla f \in (L^{\textsf{r}}(\Omega))^n$,
and
$a:=\frac{\frac{1}{ \mathsf{q}}-\frac{1}{ \mathsf{p}}}
		{\frac{1}{ \mathsf{q}}+\frac{1}{n}-
		\frac{1}{ \mathsf{r}}} \in [0,1]$. 	
\begin{proof}
Following from the Gagliardo--Nirenberg inequality 
(see \cite{Nir} for more details): 
\begin{equation*}\label{GNfirst} 	
		\|f\|^{ \mathsf{p}}_{L^{ \mathsf{p}}(\Omega)}\leq \Big[c_{{\rm GN}} \Big(\|\nabla f\|^{a}_{L^{ \mathsf{r}}(\Omega)}
		\|f\|_{L^{ \mathsf{q}}(\Omega)}^{1-a} +\|f\|_{L^{ \mathsf{s}}(\Omega)}\Big) \Big]^{ \mathsf{p}},
	\end{equation*}
with some $c_{{\rm GN}}>0$, and then from the inequality
\begin{equation*}
(\mathsf{a}+\mathsf{b})^{\alpha} \leq 2^{\alpha}(\mathsf{a}^{\alpha} + 
\mathsf{b}^{\alpha})\quad {\rm for\ any}\ \mathsf{a}, \mathsf{b}\geq 0,
\ \alpha>0,
\end{equation*}
we arrive to \eqref{GN ineq} with $C_{\rm GN}= 2^{\mathsf{p}} c_{\rm GN}^{\mathsf{p}}$.
\end{proof}
\end{lem}


%


\section{Finite-time blow-up in \boldmath{$L^\infty$}-norm} \label{Sec3}

Throughout the sequel, we suppose that $\Omega=B_R(0) \subset \R^n$ 
$(n \ge 3)$ with $R>0$ and 
$u_0$ satisfies condition \eqref{initial} as well as $\lambda \in \R$, $\mu>0$, $k>1$, 
$\chi, \xi, \alpha, \beta, \gamma, \delta>0$. 
Then we denote by $(u, v, w)=(u(r,t), v(r,t), w(r,t))$ the local classical solution of \eqref{sys1} given in Lemma~\ref{LSE} 
and by $\Tmax \in (0, \infty)$ its maximal existence time. 

The goal of this section is to prove finite-time blow-up 
in $L^\infty$-norm. 
To this end, noting that $u_0$ is radially symmetric and 
so are $u, v, w$, we first define the functions
%
\begin{align*}
U(s,t)&:=\int_0^{s^{\frac{1}{n}}} \rho^{n-1}u(\rho, t)\,d\rho,\quad s \in [0, R^n],\ t \in [0, \Tmax),\\
V(s,t)&:=\int_0^{s^{\frac{1}{n}}} \rho^{n-1}v(\rho, t)\,d\rho,\quad s \in [0, R^n],\ t \in [0, \Tmax),\\
W(s,t)&:=\int_0^{s^{\frac{1}{n}}} \rho^{n-1}w(\rho, t)\,d\rho, \quad s \in [0, R^n],\ t \in [0, \Tmax).
\end{align*}
%
Then we prove the following lemma.
\begin{lem}
Under the above notation, we have
%
\begin{align}\label{UVWeq}
U_t(s,t)
&=n^2s^{2-\frac{2}{n}}U_{ss}(s,t)
+n\chi\alpha U(s,t)U_s(s,t)
-n\chi\beta V(s,t)U_s(s,t)\notag\\[2mm]
&\quad\,
-n\xi\gamma U(s,t)U_s(s,t)
+n\xi\delta W(s,t)U_s(s,t)\notag\\
&\quad\,+\lambda U(s,t)
-n^{k-1}\mu\int_0^sU_s^k (\sigma,t)\,d\sigma
\end{align}
for all $s \in (0, R^n)$, $t \in (0, \Tmax).$
%
\end{lem}

\begin{proof}
By the definitions of $U, V, W$, we obtain
%
\begin{align*}
&U_s(s,t)=\frac{1}{n} u(s^{\frac{1}{n}}, t),\quad 
U_{ss}(s,t)=\frac{1}{n^2}s^{\frac{1}{n}-1}u_r(s^{\frac{1}{n}}, t),\\[1mm]
&V_s(s,t)=\frac{1}{n} v(s^{\frac{1}{n}}, t),\quad 
V_{ss}(s,t)=\frac{1}{n^2}s^{\frac{1}{n}-1}v_r(s^{\frac{1}{n}}, t),\notag\\[1mm]
&W_s(s,t)=\frac{1}{n} w(s^{\frac{1}{n}}, t),\quad 
W_{ss}(s,t)=\frac{1}{n^2}s^{\frac{1}{n}-1}w_r(s^{\frac{1}{n}}, t),\notag
\end{align*}
%
for all $s \in (0, R^n)$, $t \in (0, \Tmax)$. 
Since $u, v, w$ are radially symmetric functions, 
we see from the second and third equations in \eqref{sys1} that 
%
\begin{align*}
\frac{1}{r^{n-1}}(r^{n-1}v_r(r, t))_r=-\alpha u(r, t)+\beta v(r, t),\\
\frac{1}{r^{n-1}}(r^{n-1}w_r(r, t))_r=-\gamma u(r, t)+\delta w(r, t), 
\end{align*}
%
from which we obtain
%
\begin{align}
\label{vr}
r^{n-1}v_r(r, t)=-\alpha U(r^n, t)+\beta V(r^n, t),\\[1mm]
\label{wr}
r^{n-1}w_r(r, t)=-\gamma U(r^n, t)+\delta W(r^n, t)
\end{align}
%
for all $r \in (0, R)$, $t \in (0, \Tmax)$. 
Moreover, rewriting the first equation in \eqref{sys1} in the radial coordinates as
%
\begin{align*}
u_t(r,t)
&=\frac{1}{r^{n-1}}(r^{n-1}u_r(r,t))_r
    -\chi \frac{1}{r^{n-1}}(u(r,t)r^{n-1}v_r(r,t))_r\\[1mm]
&\quad\,
    +\xi \frac{1}{r^{n-1}}(u(r,t)r^{n-1}w_r(r,t))_r\\[1mm]
&\quad\,
    +\lambda u(r,t)
    -\mu u^k(r,t)
\end{align*}
%
and
integrating it with respect to $r$ over $[0, s^{\frac{1}{n}}]$, we have 
%
\begin{align*}
U_t(s,t)
&=n^2s^{2-\frac{2}{n}}U_{ss}(s,t)
-n\chi U_s(s,t)s^{1-\frac{1}{n}}v_r(s^\frac{1}{n},t)
\\[1.5mm]
&\quad\,
+n\xi U_s(s,t)s^{1-\frac{1}{n}}w_r(s^\frac{1}{n},t)\\[1mm]
&\quad\,+\lambda U(s,t)
-n^{k-1}\mu\int_0^sU_s^k (\sigma,t)\,d\sigma
\end{align*}
%
for all $s \in (0, R^n)$, $t \in (0, \Tmax)$.
Thanks to \eqref{vr} and \eqref{wr}, we arrive at \eqref{UVWeq}.
\end{proof}

Given $p \in (0,1)$, $s_0 \in (0, R^n)$, 
we next derive a differential inequality 
for the moment-type function $\Phi$ defined as
%
\begin{align*}
\Phi(t):=\int_0^{s_0} s^{-p}(s_0-s)U(s,t)\,ds, 
\quad t \in [0,\Tmax).
\end{align*}
%

%
\begin{lem}\label{DIlem}
Let $\lambda \in \R$, $\mu>0$, 
$\chi, \xi, \alpha, \beta, \gamma, \delta>0$ 
and let $\chi\alpha-\xi\gamma>0$. 
Assume that $k>1$ satisfies \eqref{condik}. 
Then there is $p\in(1-\frac{2}{n},1)$ with 
the following property\/{\rm :} 
For all $m>0$ and $L>0$ there exist $s_*\in(0,R^n)$ and $C_1>0$ such that 
whenever $u_0$ fulfills \eqref{initial}, \eqref{u0sing} and 
$\int_\Omega u_0(x)\,dx\le m$, for any $s_0\in(0,s_*)$
the function $\Phi$ satisfies 
%
\begin{align}\label{DI}
\Phi'(t) \ge \frac{1}{C_1}s_0^{p-3}\Phi^2(t)-C_1s_0^{\frac{2}{n}+1-p}
\end{align}
%
for all $t \in (0,\widehat{T}_{\max})$, 
where $\widehat{T}_{\max}:=\min\{1,\Tmax\}$. 
Moreover, for all $m_0 \in (0, m)$ 
one can find $s_0 \in (0, s_*)$ and 
$r_0=r_0(R, \lambda, \mu, k, L, m, m_0) \in (0, R)$  
such that if 
$\int_{B_{r_0}(0)}u_0(x)\,dx \ge m_0$ and 
$\widehat{T}_{\max}>\frac{1}{2}$, 
then for all $t \in (0,\frac{1}{2})$,  
%
\begin{align}\label{DI2}
\Phi'(t) \ge C_2s_0^{p-3}\Phi^2(t),
\end{align}
%
where $C_2$ is a positive constant. 

\end{lem}
%

\begin{proof}
By the definition of $\Phi$ and 
equation \eqref{UVWeq}, we have
%
\begin{align}\label{DPhi} 
\Phi'(t) &=\int^{s_0}_{0}s^{-p}(s_0-s)U_t(s,t)\,ds\notag\\
           &=n^2\int^{s_0}_{0}s^{2-\frac{2}{n}-p}(s_0-s)U_{ss}(s,t)\,ds\notag\\
           &\quad\,
             +n(\chi\alpha-\xi\gamma) \int^{s_0}_{0}
             s^{-p}(s_0-s)U(s,t)U_s(s,t)\,ds\notag\\
           &\quad\,
             -n\chi\beta \int^{s_0}_{0}s^{-p}(s_0-s)V(s,t)U_s(s,t)\,ds\notag\\ 
           &\quad\,
             +n\xi\delta \int^{s_0}_{0}s^{-p}(s_0-s)W(s,t)U_s(s,t)\,ds\notag\\ 
           &\quad\,
             +\lambda\int^{s_0}_{0}s^{-p}(s_0-s)U(s,t)\,ds
             -n^{k-1}\mu\int^{s_0}_{0}s^{-p}(s_0-s)
             \Big[ \int^{s}_{0}U_s^{k}(\sigma,t)\,d\sigma \Big]\,ds.
\end{align}
%
Since $U_s(s,t)=\frac{1}{n} u(s^{\frac{1}{n}}, t) \ge  0$ and hence the fourth term on the right-hand side of \eqref{DPhi} is nonnegative, we obtain
%
\begin{align*}
\Phi'(t) &\ge n^2\int^{s_0}_{0}s^{2-\frac{2}{n}-p}(s_0-s)U_{ss}(s,t)\,ds\\ \notag
  &\quad\,+n(\chi\alpha-\xi\gamma) \int^{s_0}_{0}s^{-p}(s_0-s)U(s,t)U_s(s,t)\,ds\\ \notag
  &\quad\,-n\chi\beta \int^{s_0}_{0}s^{-p}(s_0-s)V(s,t)U_s(s,t)\,ds\\ \notag
  &\quad\,+\lambda\int^{s_0}_{0}s^{-p}(s_0-s)U(s,t)\,ds
-n^{k-1}\mu\int^{s_0}_{0}s^{-p}(s_0-s)
                  \Big[\int^{s}_{0}U_s^k(\sigma,t)\,d\sigma \Big]\,ds
\end{align*}
%
for all $t \in (0, \Tmax)$. 
Since $\chi\alpha-\xi\gamma>0$ by assumption, following the steps in \cite[(4.3)]{W-2018}, 
we can derive the differential inequalities \eqref{DI} and \eqref{DI2}; 
note that, in the assumption $\widehat{T}_{\max}>\frac{1}{2}$ for \eqref{DI2} 
the value $\frac{1}{2}$ can be replaced 
with other positive values less than $1$.
\end{proof}

Now, we can prove Theorem~\ref{BULinfty}. 

\begin{prth1.1}
Thanks to Lemma~\ref{DIlem}, in particular, from \eqref{DI2}, we can see that 
$\Tmax <\infty$. 
Therefore, from blow-up criterion  \eqref{bucri}, we conclude that the finite-time blow-up in $L^\infty$-norm occurs. 
Namely, \eqref{blowupinfty} is proved. \qed
\end{prth1.1}


\section{{Finite-time blow-up in \boldmath{$L^\sigma$}-norm}} \label{Sec4}

\noindent

In these next sections we will assume the conditions contained in 
Theorem~\ref{BULinfty}.
In order to prove Theorem \ref{BULsig}, first we state the following lemmas.
\begin{lem}\label{lemma int u < bar m}
Let\/ $\Omega \subset \mathbb{R}^n,\ n\geq 3$ 
be a bounded and smooth domain, and $\lambda \in \R$, $\mu >0$, $k>1$. 
Then for a classical solution $(u, v, w)$ of\/ \eqref{sys1} we have 
\begin{equation}\label{int u < bar m}
\int_{\Omega} u\,dx \leq m_* \quad \textrm{for all} \ t \in(0,\Tmax),
\end{equation}
with
\begin{equation}\label{ m}
m_*:= \max \Big \{\int_{\Omega} u_0\,dx, \ 
\Big(\frac{\lambda_+}{\mu}|\Omega |^{k-1}\Big)^{\frac{1}{k-1}}\Big \},
\end{equation}
where $\lambda_+:=\max\{0, \lambda\}$.
\end{lem}

\begin{proof}
Integrating the first equation in \eqref{sys1} and applying the divergence theorem and boundary conditions of \eqref{sys1}, we obtain
\begin{equation}\label{max princ m}
\frac{d}{dt}\int_{\Omega} u \,dx
= \lambda \int_{\Omega} u \,dx - \mu \int_{\Omega} u ^ k\,dx 
\leq \lambda_+ \int_{\Omega} u\,dx 
- \mu |\Omega |^{1-k}\Big (\int_{\Omega} u \,dx \Big)^{k},
\end{equation}
where in the last term we used H$\ddot{\rm o}$lder's inequality: $\int_{\Omega} u \leq |\Omega|^{\frac{k-1}{k}} (\int_{\Omega} u^k)^{\frac 1 k}$.
From \eqref{max princ m} we deduce that 
$y:=\int_{\Omega}u\,dx$ fulfills 
 \begin{equation*}
\begin{cases}
y' (t) \leq \lambda_+ y(t) - \bar{\mu} y^k(t), \ \ \ \ \bar{\mu} :=\mu |\Omega |^{1-k} \ \ \ \textrm{for all} \ t\in (0,\Tmax),\\[6pt]
y(0)=y_0, \ \ \ \ y_0:=\int_\Omega u_0\,dx.
\end{cases}
\end{equation*}
Upon an ODE comparison argument this implies that
\begin{equation*} \label{y < bar m}
y(t)\leq m_*\quad 
\textrm{for all}\ t \in(0,\Tmax).
\end{equation*}
The lemma is proved.
\end{proof}

We next prove the following lemma which plays an important role 
in the proof of Theorem~\ref{BULsig}.

\begin{lem}\label{LemmaBoundedness u nabla v} 
Let\/ $\Omega \subset\mathbb{R}^n,\ n\geq 3$ be a bounded and smooth domain. 
Let $(u,v,w)$ be a classical solution of system \eqref{sys1}. 
If for some 
$\sigma_0>\frac{n}{2} $
there exists 
$C>0$
such that 
\begin{align*}
\left\|u(\cdot,t)\right\|_{
L^{\sigma_0}(\Omega)} 
\leq 
C
\quad \textrm{ for all } t \in (0,\Tmax),
\end{align*}
then, for some $\hat C>0$, 
\begin{align}\label{BoundednessU^infty}
\left\|u(\cdot,t)\right\|_{L^{\infty}(\Omega)} 
\leq {\hat C} \quad \textrm{ for all } t \in (0,\Tmax).
\end{align}
\end{lem}

\begin{proof}
For any $x\in \Omega$, $t\in (0,\Tmax)$, we set $t_0 := \max \{0, t-1\}$ and we consider the representation formula for $u$:
\begin{align*}
u(\cdot ,t) &= e^{(t-t_0)\Delta} u( \cdot, t_0) - \chi \int_{t_0}^{t} e^{(t-s)\Delta} \nabla \cdot (u (\cdot, s)\nabla v(\cdot,s))\,ds\notag\\
&\quad\,+\xi \int_{t_0}^{t} e^{(t-s)\Delta} \nabla \cdot (u (\cdot, s)\nabla w(\cdot,s)) \,ds+ \int_{t_0}^{t} e^{(t-s) \Delta } \big[ \lambda u(\cdot,s) - \mu u^k(\cdot,s) \big]\,ds\notag\\
&=: u_1(\cdot,t)+u_2(\cdot,t)+u_3(\cdot,t) + u_4(\cdot, t)
\end{align*}
and
\begin{align} \label{unorm}
0\leq u(\cdot ,t)\leq \| u_1(\cdot,t) \|_{L^{\infty}(\Omega)} +\|u_2(\cdot,t)\|_{L^{\infty}(\Omega)}
+\|u_3(\cdot,t) \|_{L^{\infty}(\Omega)}+u_4(\cdot, t).
\end{align}
We have
\begin{equation}\label{u_1}
\begin{split}                                                                                                                                                                                                                                                                                                                                                                                                                                                                                                                                                                                                                                                                                                                                                                                                                                                                                                                                                                                                                                                                                                                                                                                                                                                                                                                                                                                                                                                                                                                                                                                                                                                                                                                                                                                                                                                                                                                                                                                                                                                                                                                                                                                                                                                                                                                                                                                                                                                                                                                                                                                                                                                                                                                                                                                                                                                                                                                                                                                                                                                                     \lVert u_1 (\cdot,t) \rVert _{L^{\infty}(\Omega)} \leq \max \{\lVert u_0 \rVert_{L^{\infty}(\Omega)}, m_* k_1\} =:C_5,
\end{split}
\end{equation}
with $k_1>0$ and $m_*$ defined in \eqref{ m}. 
In fact, if $t\leq 1$, then $t_0=0$ and hence the maximum principle yields $u_1(\cdot, t) \leq \| u_0\|_{L^{\infty}(\Omega)}$. 
If $t>1$, then $t-t_0=1$ and from \eqref{etDeltaz} with $p=\infty$ and $q=1$, we deduce from \eqref{int u < bar m} that $\lVert u_1(\cdot,t)\rVert_{L^{\infty}(\Omega)} \leq k_1 (t-t_0)^{-\frac n 2} \lVert u(\cdot,t_0) \rVert_{L^1(\Omega)} \leq m_* k_1$.
We next use \eqref{etDelta nablaz} with $p=\infty$, which leads to
\begin{align}\label{u_2} 
\|u_2 (\cdot, t) \|_{L^{\infty}(\Omega)}
&\leq k_2 \chi  \int_{t_0}^{t} ( 1 + (t-s)^{-\frac 1 2 - \frac{n}{2q} }) e^{-\mu_1  (t-s)} \|u (\cdot, s)\nabla v(\cdot,s)  \|_{L^q(\Omega)} \,ds.
\end{align}
Here, we may assume that 
$\frac{n}{2}<\sigma_0<n$,
and then 
we can fix $q>n$ such that 
$1-\frac{(n-\sigma_0)q}{n\sigma_0}>0$,
which enables us to pick $\theta \in (1,\infty)$ fulfilling
$\frac{1}{\theta}<1-\frac{(n-\sigma_0)q}{n\sigma_0}$,
that is, 
$\frac{q\theta}{\theta-1}<\frac{n\sigma_0}{n-\sigma_0}$.
Then by H$\ddot{{\rm o}}$lder's inequality, we can estimate
\begin{align*}
\|u (\cdot, s)\nabla v(\cdot,s)  \|_{L^q(\Omega)} 
&\le 
\|u (\cdot, s)\|_{L^{q\theta}(\Omega)}\|\nabla v(\cdot,s)  \|_{L^{\frac{q\theta}{\theta-1}}(\Omega)}\\
&\le C_6\|u (\cdot, s)\|_{L^{q\theta}(\Omega)} \|\nabla v(\cdot,s)\|_{L^{
\frac{n\sigma_0}{n-\sigma_0}}(\Omega)}\quad 
{\rm for\ all}\ s \in (0, \Tmax),
\end{align*}
with some $C_6>0$. 
The Sobolev embedding theorem and elliptic regularity theory 
applied to the second equation in \eqref{sys1} tell us that 
$\|v(\cdot,s)\|_{W^{
1,\frac{n\sigma_0}{n-\sigma_0}}(\Omega)}
\leq C_7\|v(\cdot,s)\|_{W^{
2,\sigma_0}(\Omega)}
\le C_8$ with some $C_7, C_8>0$. 
Thus again by H$\ddot{{\rm o}}$lder's inequality and \eqref{int u < bar m}, we obtain
\begin{align*}
\|u (\cdot, s)\nabla v(\cdot,s)  \|_{L^q(\Omega)} 
&\le 
C_9\|u (\cdot, s)\|_{L^{q\theta}(\Omega)}\\
&\le C_{10}\|u (\cdot, s)\|_{L^{\infty}(\Omega)}^{\bar \theta}\quad 
{\rm for\ all}\ s \in (0, \Tmax),
\end{align*}
with some $\bar \theta \in (0,1)$, $C_9:=C_6C_8$ and $C_{10}:=C_9m_*^{1-\bar\theta}$. 
Hence, combining this estimate and \eqref{u_2}, we infer
\begin{align*}
\|u_2 (\cdot, t) \|_{L^{\infty}(\Omega)}
\leq C_{10}k_2 \chi   \int_{t_0}^{t} ( 1 +  (t-s)^{-\frac 1 2 - \frac{n}{2q} }) e^{-\mu_1  (t-s)} \|u (\cdot, s)\|_{L^{\infty}(\Omega)}^{\bar \theta}\,ds.
\end{align*}
Now fix any $T \in (0, \Tmax)$. 
Then, since $t-t_0\leq 1$, we have
\begin{align}\label{u_21}
\|u_2 (\cdot, t) \|_{L^{\infty}(\Omega)}
&\leq C_{10}k_2 \chi   \int_{t_0}^{t} ( 1 +  (t-s)^{-\frac 1 2 - \frac{n}{2q} } e^{-\mu_1  (t-s)}) \,ds \cdot \sup_{t \in [0, T]} \|u (\cdot, t)\|_{L^{\infty}(\Omega)}^{\bar \theta}\notag\\
&\leq C_{11}\chi\sup_{t \in [0, T]} \|u (\cdot, t)\|_{L^{\infty}(\Omega)}^{\bar \theta},
\end{align}
where $C_{11}:=C_{10}k_2(1+\mu_1^{\frac{n}{2q}-\frac{1}{2}}\int_0^\infty r^{-\frac 1 2 - \frac{n}{2q}} e^{-r}\,dr)>0$ is finite, 
because $-\frac{1}{2}-\frac{n}{2q}>-1$. 
Analogously we can conclude
\begin{align}\label{u_3}
\|u_3 (\cdot, t) \|_{L^{\infty}(\Omega)}
&\le C_{11}\xi \sup_{t \in [0, T]} \|u (\cdot, t)\|_{L^{\infty}(\Omega)}^{\bar \theta}.
\end{align}
We next prove that there exists a constant $C_{12}\ge0$ such that $u_4(\cdot, t)\leq C_{12}$. To this end, we firstly observe that 
 $$h(u):= \lambda u - \mu u^k \leq h(u_*) =:C_{12},$$ 
 with $u_*:= (\frac{\lambda_{+}}{\mu k})^{\frac{1}{k-1}}$. 
We have
\begin{align}\label{u_4}
u_4(\cdot, t) &=  \int_{t_0}^{t} e^{(t-s) \Delta } \big[ \lambda u(\cdot,s) - \mu u^k(\cdot,s) \big]\,ds\leq C_{12} \int_{t_0}^t \,ds\leq C_{12}.
\end{align}
Plugging \eqref{u_1}, \eqref{u_21}, \eqref{u_3} and \eqref{u_4} 
into \eqref{unorm},	we see that
\begin{align*}
0 \le u(x, t) \le C_5+C_{12}+C_{11}(\chi+\xi) \sup_{t \in [0, T]} \|u (\cdot, t)\|_{L^{\infty}(\Omega)}^{\bar \theta},
\end{align*}
which implies
\begin{align*}
\sup_{t \in [0, T]}\|u(\cdot, t)\|_{L^\infty(\Omega)} 
&\le C_{13}+C_{14} \Big(\sup_{t \in [0, T]} \|u (\cdot, t)\|_{L^{\infty}(\Omega)}\Big)^{\bar \theta}\quad {\rm for\ all}\ T \in (0, \Tmax),
\end{align*} 
with $C_{13}:=C_5+C_{12}$ and $C_{14}:=C_{11}(\chi+\xi)$. 
From this inequality with $\bar \theta \in (0,1)$, we arrive at \eqref{BoundednessU^infty}.
\end{proof}

\begin{prth1.2}
Since Theorem~\ref{BULinfty} holds, the unique local classical solution of \eqref{sys1} blows up at $t=\Tmax$ in the sense 
$\limsup_{t \nearrow \Tmax}\| u(\cdot , t) \|_{L^{\infty}(\Omega)}= \infty$ (i.e., \eqref{blowupinfty}).
By contradiction, we prove that it blows up also in $L^{\sigma}$-norm. 
In fact, if there exist $\sigma_0> \frac n 2 $ and $C>0$ 
such that 
\begin{align*}
\|u(\cdot, t)\|_{L^{\sigma_0}(\Omega)} \leq C\quad 
{\rm for\ all}\ t \in (0, \Tmax),
\end{align*} 
then, from Lemma \ref{LemmaBoundedness u nabla v}, there exists 
$\hat C>0$ such that 
\begin{equation*} 
\lVert u(\cdot,t)\rVert _{L^\infty(\Omega)}\leq \hat C\quad
{\rm for\ all}\ t \in (0, \Tmax),
\end{equation*}
which is in contradiction to \eqref{blowupinfty}, so that, 
if $u$ blows up in $L^{\infty}$-norm, 
then $u$ blows up also in $L^{\sigma}$-norm 
for all $\sigma>\frac{n}{2}$. \qed
\end{prth1.2}



\section{A lower bound for \boldmath{$\Tmax$}, 
the proof of Theorem \ref{LB}} \label{Sec5}
Let us consider $\Psi(t)= \frac {1}{\sigma} \int_{\Omega} u^{\sigma}(x,t)\,dx$, 
$u(x,t)$ the first component of solutions to \eqref{sys1} and we prove that $\Psi$ satisfies a first order differential inequality.

In the proof of Theorem \ref{LB} we need an estimate for $\int_{\Omega} u^{\sigma + 1}\, dx$. 
To this end, we use the Gagliardo--Nirenberg 
inequality  \eqref{GN ineq} with $f=u^{\frac{\sigma}{2}}$, $\mathsf{p}=\frac{2(\sigma +1)}{\sigma}$,  $\mathsf{r}=2, \mathsf{q}=2, \mathsf{s}= 2$. 
Since $\sigma>\frac n 2$, we have
\begin{align} \label{GN_u^sig+1}
\int_{\Omega} u^{\sigma + 1} \,dx
&=\|u^{\frac{\sigma}{2}}\|_{L^{\frac{2(\sigma +1)}{\sigma}}(\Omega)}^{\frac{2(\sigma + 1)}{\sigma}}\notag\\
&\leq C_{{\rm GN}} \|\nabla u^{\frac {\sigma}{2}}\|^{\frac{2(\sigma +1)}{\sigma} \theta_0}_{L^2(\Omega)} \| u^{\frac{\sigma}{2}} \|^{\frac{2(\sigma +1)}{\sigma}(1-\theta_0)} _{L^2(\Omega)}+C_{{\rm GN}}\|u^{\frac{\sigma}2}\|^{\frac{2(\sigma +1)}{\sigma}}_{L^{2}(\Omega)}\notag\\
&=C_{{\rm GN}}\Big(\int_{\Omega} |\nabla u^{\frac{\sigma}{2}}|^2 \,dx\Big)^{{\frac{\sigma +1}{\sigma}}\theta_0}\Big(\int_{\Omega}u^{\sigma} \,dx\Big)^{\frac{\sigma +1}{\sigma}(1-\theta_0)}+
C_{{\rm GN}} \Big(\int_{\Omega} u^{\sigma} \,dx\Big)^{\frac{\sigma+1}{\sigma}} \notag\\
&\leq C_{{\rm GN}} \varepsilon_1 \beta_0 \int_{\Omega}|\nabla u^{\frac{\sigma}{2}}|^2 \,dx+
C_{{\rm GN}} \varepsilon_1^{-\frac{\beta_0}{1-\beta_0}}(1-\beta_0)  \Big(\int_{\Omega} u^{\sigma} \,dx\Big)^{\frac{(\sigma+1)(1-\theta_0)}{\sigma(1-\beta_0)}}\notag\\
&\quad\,+ C_{{\rm GN}}  \Big(\int_{\Omega} u^{\sigma} \,dx\Big)^{\frac{\sigma +1}{\sigma}}\notag\\
&=c_1(\varepsilon_1)\int_{\Omega}|\nabla u^{\frac{\sigma}{2}}|^2 \,dx + c_2(\varepsilon_1) \Big(\int_{\Omega} u^{\sigma} \,dx\Big)^{\frac{2(\sigma +1)-n}{2\sigma -n}} + c_3\Big(\int_{\Omega} u^{\sigma} \,dx\Big)^{\frac{\sigma +1}{\sigma}}, 
\end{align}
with $\varepsilon_1>0$, $\theta_0:=\frac{n}{2(\sigma+1)} \in (0,1)$ and $\beta_0:= \frac{\sigma + 1}{\sigma} \theta_0 =\frac{n}{2\sigma}\in(0,1)$.
 Now, we derive a differential inequality of the first order for $\Psi(t)$. 
\begin{align}
 \label{psi'}
\Psi'(t)&= \int_\Omega u^{\sigma -1} \Delta u \,dx
- \chi\int_\Omega u^{\sigma -1}\nabla\cdot (u \nabla v )\,dx
+ \xi \int_\Omega u^{\sigma -1}\nabla \cdot(u \nabla w )\,dx\notag\\
&\quad\,+\lambda \int_\Omega u^{\sigma}\,dx- \mu \int_\Omega u^{\sigma +k-1}\,dx\notag\\
&=:\mathcal  I_1+  \mathcal I_2 + \mathcal  I_3+ \mathcal  I_4+ \mathcal I_5.
\end{align}
We have:
\begin{align}
\mathcal  I_1&=\int_{\Omega} u^{\sigma-1} \Delta u \,dx 
=- (\sigma -1) \int_{\Omega} u^{\sigma-2} |\nabla u|^2 \,dx \notag\\
&= -\frac{4(\sigma - 1)}{\sigma^2} \int_{\Omega} 
|\nabla u^{\frac{\sigma}{2}}|^2\,dx,\label{I1}
\intertext{and}
\mathcal  I_2 &=- \chi  \int_{\Omega} u^{\sigma-1} \nabla\cdot (u \nabla v)\, dx 
= \chi \frac{\sigma -1}{\sigma}  \int_{\Omega} \nabla u^{\sigma} \cdot\nabla v \,dx \notag\\
&=
- \chi \frac{\sigma -1}{\sigma}  \int_{\Omega} u^{\sigma} \Delta v \,dx \notag\\
&= - \chi \beta \frac{\sigma -1}{\sigma}  \int_{\Omega} u^{\sigma}  v \,dx + \chi \alpha \frac{\sigma -1}{\sigma}  \int_{\Omega} u^{\sigma +1}\,dx \notag\\
&\leq \chi \alpha \frac{\sigma -1}{\sigma}  \int_{\Omega} u^{\sigma +1}\,dx \label{I2}
\end{align}
as well as
\begin{align}
\mathcal  I_3&=\xi \int_\Omega u^{\sigma -1}\nabla \cdot(u \nabla w )\,dx\notag\\[1mm]
&= \xi \delta \frac {\sigma -1}{\sigma}  \int_\Omega u^{\sigma} w \,dx - \xi \gamma \frac {\sigma -1} {\sigma}  \int_\Omega u^{\sigma + 1}\,dx \notag\\[1mm]
&\leq
\xi \delta \frac {\sigma -1} {\sigma} \Big(  \int_\Omega u^{\sigma +1} \,dx\Big)^{\frac{\sigma}{\sigma + 1 }}   \Big(  \int_\Omega w^{\sigma +1} \,dx\Big)^{\frac{1}{\sigma + 1 }}- \xi \gamma \frac {\sigma -1} {\sigma}  \int_\Omega u^{\sigma + 1}\,dx \notag\\[1mm]
&\leq \xi \gamma \frac {\sigma -1} {\sigma} \int_\Omega u^{\sigma +1} \,dx- \xi \gamma \frac {\sigma -1} {\sigma}  \int_\Omega u^{\sigma + 1}\,dx\notag\\[1mm]
&=0\label{I3},
\end{align}
where the last inequality holds from $(\int_\Omega w^{\sigma+1} )^{\frac{1}{\sigma+1}}
\le \frac{\gamma}{\delta}(\int_\Omega u^{\sigma+1})^{\frac{1}{\sigma+1}}$ established by standard testing procedures in the equation for $w$. 
We now use \eqref{GN_u^sig+1} in \eqref{I2} to obtain 
\begin{equation}\label{I2+I3 bis}
\begin{split}
\mathcal  I_2 \leq  \tilde c_1(\varepsilon_1)\int_{\Omega}|\nabla u^{\frac{\sigma}{2}}|^2 \,dx + \tilde c_2(\varepsilon_1)\Big(\int_{\Omega} u^{\sigma} \,dx\Big)^{\frac{2(\sigma +1)-n}{2\sigma -n}} + \tilde c_3 \Big(\int_{\Omega} u^{\sigma} \,dx\Big)^{\frac{\sigma +1}{\sigma}},
\end{split}
\end{equation}
with $\tilde c_1(\varepsilon_1):=\chi\alpha\frac {\sigma -1}{\sigma} c_1(\varepsilon_1), \ \tilde c_2(\varepsilon_1):=\chi\alpha\frac {\sigma -1}{\sigma} c_2(\varepsilon_1), \ \tilde c_3:=\chi\alpha\frac {\sigma -1}{\sigma} c_3$. 
Also, using H$\ddot{{\rm o}}$lder's inequality, we see that
\begin{align}\label{I4+I5}
\mathcal  I_4+\mathcal  I_5 &= \lambda \int_\Omega u^{\sigma}\,dx -   \mu 
\int_\Omega u^{\sigma +k-1} \,dx \notag\\
&\leq \lambda_{+} \int_\Omega u^{\sigma}\,dx -
\mu |\Omega|^{\frac{1-k}{\sigma}}
\Big (\int_\Omega u^{\sigma} \,dx\Big)^{\frac{\sigma +k-1}{\sigma} }.
\end{align}
Substituting \eqref{I1}, \eqref{I3}, \eqref{I2+I3 bis}  and \eqref{I4+I5} in \eqref{psi'} we 
get
\begin{align}\label{psi'bis}
\Psi'&\leq B_1 \Psi + B_2 \Psi^{\frac{\sigma + 1}{\sigma}}
+ B_3 \Psi^{\frac{2(\sigma +1)-n}{2\sigma -n}} 
- B_4 \Psi^{\frac{\sigma +k-1}{\sigma}}\notag\\
&\quad\,+\Big(\tilde c_1(\varepsilon_1)- \frac{4(\sigma-1)}{\sigma^2}  \Big)\int_\Omega |\nabla u^{\frac{\sigma}{2}}|^2\,dx,
\end{align}
with $B_1:= \lambda_{+} \sigma, \ B_2:=\tilde c_3 \sigma^{\frac{\sigma + 1}{\sigma}}, \ B_3:=\tilde c_2(\varepsilon_1) \sigma^{\frac{2(\sigma +1)-n}{2\sigma -n}} , \ B_4:= \mu |\Omega|^{\frac{1-k}{\sigma}} { {\sigma}^{\frac{\sigma +k-1}{\sigma} }}$.
In \eqref{psi'bis} we choose $\varepsilon_1$ such that 
$\tilde c_1(\varepsilon_1)- \frac{4(\sigma-1)}{\sigma^2}\leq 0$ and neglecting the negative terms, we obtain 
\begin{align}\label{psi'final}
\Psi'\leq B_1 \Psi + B_2 \Psi^{\frac{\sigma + 1}{\sigma}}
+ B_3 \Psi^{\frac{2(\sigma +1)-n}{2\sigma -n}}.
\end{align}
Integrating \eqref{psi'final} from $0$ to $\Tmax$, we arrive to \eqref{lower Tmax in Lp}.
\hfill \qed

\begin{remark} 
Since $u$ blows up in $L^{\sigma}(\Omega)$-norm 
at finite time $\Tmax$, then there exists a time 
$t_1\in [0,\Tmax)$, where $\Psi(t_1)=\Psi_0$. 
As a consequence, $\Psi(t) \geq \Psi_0$, $t\in [t_1, \Tmax)$ 
so that $\Psi^{\rho} \leq \Psi^{\gamma_2} \Psi_0^{\rho-\gamma_2}$ 
for some $\rho \leq \gamma_2$. 
Moreover, taking into account that $1< \frac{\sigma + 1}{\sigma} \leq \frac{2(\sigma +1)-n}{2\sigma -n}=\gamma_2$, it follows that 
\begin{align}\label{psi'bisf}
\Psi'\leq  A \Psi^{\gamma_2} \quad {\rm in} \  (t_1,\Tmax),
\end{align}
with $ A:= B_1 \Psi_0^{-\frac{2}{2\sigma -n}} + B_2 \Psi_0^{-\frac{n}{\sigma(2\sigma -n)}} + B_3$.
Integrating \eqref{psi'bisf} from $t_1$ to $\Tmax$, 
we derive the following explicit lower bound of the blow-up time 
$\Tmax$: 
\begin{equation*}
\Tmax \geq \frac{1}{A (\gamma_2 -1)\Psi_0^{\gamma_2 -1}}.
\end{equation*}
\end{remark}

\section*{Acknowledgments}
The authors would like to express their gratitude to 
Professor Stella Vernier-Piro for giving them the opportunity of 
a joint study and her encouragement. 
YC, YT and TY are partially supported by 
Tokyo University of Science Grant for International Joint Research. 
MM is a member of the Gruppo Nazionale per l'Analisi Matematica, la Probabilit$\grave{\rm a}$ e le loro Applicazioni (GNAMPA) of the Istituto Nazionale di Alta Matematica (INdAM) and is partially supported by the research project {\it Evolutive and stationary Partial Differential Equations with a focus on biomathematics} (Fondazione di Sardegna 2019).


\end{document}